\documentclass[11pt, reqno]{amsart}
\usepackage[]{hyperref}
\usepackage{amsmath}
\usepackage{amsfonts}
\usepackage{amssymb}
\usepackage[]{overpic}
\usepackage{tikz}
	\usetikzlibrary{arrows}
	\usetikzlibrary{decorations.pathreplacing} 
	\usetikzlibrary{matrix}
\usepackage{stmaryrd}
\usepackage{multirow}
\usepackage{mathtools}
\usepackage{subfigure}
\usepackage{verbatim}
\usepackage{pinlabel}
\usepackage[all,cmtip,knot]{xy}
\usepackage{enumerate, enumitem, comment}
\usepackage[normalem]{ulem}

\newtheorem*{theorem}{Theorem}

\newtheorem*{proposition}{Proposition}

\newtheorem*{corollary}{Corollary}

\newtheorem*{conjecture}{Conjecture}

\theoremstyle{definition}
\newtheorem*{definition}{Definition}

\theoremstyle{remark}
\newtheorem*{remark}{Remark}

\newtheorem*{namedtheorem}{\theoremname}
\newcommand{\theoremname}{testing}

\def\Z{\mathbb{Z}}

\def\ref{\textup{ref}}

\definecolor{lblue}{HTML}{A1BAC2}
\definecolor{lgreen}{HTML}{AEBC7D}
\definecolor{lred}{HTML}{BB6A5E}

\numberwithin{equation}{section}

\author[Tye Lidman]{Tye Lidman}
\thanks{}
\address{Department of Mathematics, North Carolina State University}
\email{tlid@math.ncsu.edu}

\author[Lisa Piccirillo]{Lisa Piccirillo}
\thanks{}
\address{Department of Mathematics, The University of Texas at Austin}
\email{lisa.piccirillo@austin.utexas.edu}

\title{Stably irreducible non-orientable knotted surfaces}
\date{}

\begin{document}
\maketitle
\vspace{-12pt}
\begin{abstract}
We give an elementary obstruction to reducibility for knotted surfaces in the four-sphere.  As a new application, we construct stably irreducible non-orientable surfaces.  
\end{abstract}
\vspace{5pt}
In four-dimensional knot theory, one studies embeddings of closed (not necessarily orientable) surfaces into the four-sphere. As in the classical setting there are standard \textit{unknotted} embeddings\footnote{For every closed non-orientable surface $\Sigma$ and integer $e$ which can arise as the normal Euler number of an embedding of $\Sigma$ in $S^4$, there is an unknot $U_{\Sigma, e}$ with that surface type and normal Euler number.}
Some knotted surfaces decompose as a knotted surface connected summed with a (non-spherical) unknotted surface; surfaces with such a decomposition are called \textit{reducible}.  There are many methods in the literature for producing irreducible knotted surfaces, see for example \cite{BMS, Gordon, Litherland, Livingston, Maeda, Yoshikawa}. However, for $RP^2$, which is the simplest surface which is not tautologically irreducible, no irreducible examples have been found. 

\begin{conjecture}[Kinoshita conjecture, see \cite{KS} Remark 3.7]
Any embedded real projective plane in $S^4$ is isotopic to a connected sum of an unknotted projective plane and a knotted 2-sphere.
\end{conjecture}

We cannot disprove the Kinoshita conjecture, but we contribute a new elementary method for proving irreducibility via studying the group homology of the fundamental group of the branched double cover. We were inspired by an observation of Kervaire \cite[p.1]{Kervaire}, but note that some previous proofs of irreducibiliy have also utilized group homology in other ways.  Our method in fact establishes the following stronger notion of irreducibility; a surface $S$ is called {\em stably irreducible} if there is no surface $S'$ with $\chi(S')  > \chi(S)$ and unknotted surfaces $U, U'$ such that $S\# U$ and $S' \# U'$ are equivalent. Stably irreducible implies irreducible, the casual reader can safely ignore instances of `stably'.  Livingston constructed examples of orientable stably irreducible surfaces of any genus \cite{Livingston}, but constructing non-orientable stably irreducible surfaces seems to remain open.  

\begin{theorem}\label{prop:obstruct}
If $S$ is a double of a ribbon surface in $B^4$ with $\chi(S)= 2-2k <2$ and $rk(H_2(\pi_1(\Sigma_2(S))))=k$, then $S$ is stably irreducible. 
\end{theorem}

Here $\Sigma_2(S)$ refers to the branched double cover. The ribbon hypothesis will be used to show that the branched double cover contains many homologically essential square zero spheres.
Using the Theorem, we can easily construct stably irreducible surfaces. 

\begin{corollary}\label{thm:generalized-kinoshita-conjecture}
For any $k> 0$ there exist closed connected surfaces $S$ in $S^4$ with $\chi(S)=2-2k$ which are stably irreducible. The surfaces $S$ can be taken either orientable or not. 
\end{corollary}

Our examples have (normal) Euler number 0.  Maggie Miller pointed out it would be interesting to find normal Euler number $\pm 4$ Klein bottles which are stably irreducible.

This note is organized as follows. We first give a construction of the desired stably irreducible surfaces, assuming the Theorem.  Then we prove a weaker version of the Theorem to illustrate the ideas of the obstruction.  Then we prove the Theorem. Finally we give an additional simple construction of knotted surfaces which are stably irreducible. We work in the smooth category for convenience, but note that our methods also apply in the topological category. 

\begin{proof}[Proof of the Corollary]\label{con:examples}
To construct knotted surfaces which satisfy the hypotheses of the Theorem we will begin by choosing a suitable group $G$ which we will force to appear as $\pi_1(\Sigma_2(S))$. We begin with the case $\chi(S) = 0$. We choose a finitely-presented group $G$ with $b_2(G) = 1$ and a deficiency $-1$ presentation. There are many such groups, for example the triangle group
\[
T(2,3,7) = \langle x, y, z \mid x^2 = y^3 = z^7 = xyz = 1\rangle
\]
has $H_2(T(2,3,7)) = \mathbb{Z}$ \cite[p.576]{Triangle}.


\begin{figure}
\includegraphics[scale=.4,trim = 0in 2.6in 0in 2.3in, clip]{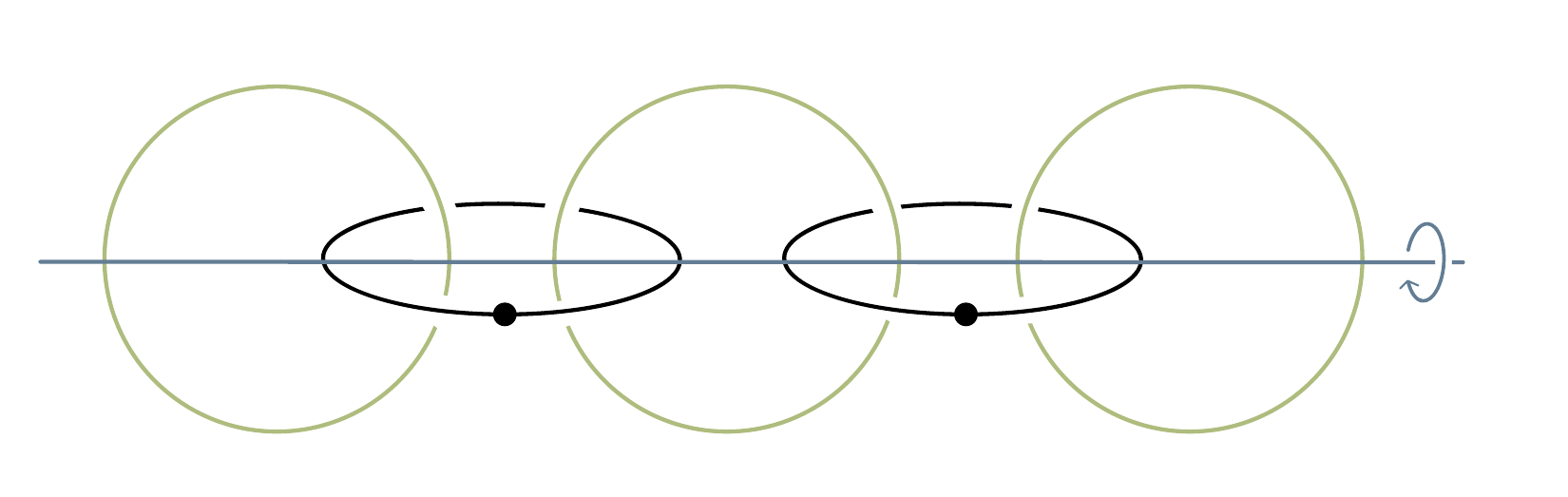}
\caption{A handle diagram for a 4-manifold $X'$ with $H_2(X') = \mathbb{Z}$ which admits an involution and has fundamental group $T(2,3,7)$ can be obtained by attaching 2-handles along cables of the green curves. In particular, from left to right, the $(7,2), (3,2)$, and $(2,1)$ cables work. Any choice of framings works. }\label{fig:smyhandle}. 
\end{figure}

To build $S$ we will build a closed four-manifold $X$ with $b_2(X) = 2$ and an involution $\varphi$ with $X/\varphi = S^4$; the branch set in $X/\varphi$ will be $S$. By construction $\Sigma_2(S) = X$.  To build X, we first build a 4-manifold $X'$ from 0,1, and 2-handles corresponding to a deficiency $-1$ presentation such that $\pi_1(X') = G$ and $H_2(X') = \mathbb{Z},$ and such that $X'$ has an involution $\varphi$ which acts on each handle by strong inversion.\footnote{We do not claim that it is always possible to build a handlebody for $G$ that admits such an involution, but it is often simple in practice.} See Figure \ref{fig:smyhandle} for an example with $G=T(2,3,7)$. The 2-handle framings do not affect the fundamental group or the involution, so we can choose them as we like; in particular we can make the intersection form of $X'$ either even or odd.  Since $\varphi$ is a strong inversion on every handle, the handles of $X'$ descend to topology of the branch set rather than topology of the quotient manifold.  Since $X'$ has no 3-handles, the branch set will be a ribbon surface in $B^4$.  We then get $X$ by doubling $X'$ and there is a (doubled) involution on $X$ which we still call $\varphi$.  So we see that $X$ is the double branched cover of a knotted surface $S$ in $S^4$, and $S$ is the double of a ribbon surface. If $X$ is spin then $S$ is a torus, and $S$ is a Klein bottle if $X$ is not spin. The normal Euler number of the Klein bottle is 0 since it is built as a double. The Theorem implies that $S$ is stably irreducible.  To obtain examples with negative Euler characteristic, observe that $\#_\ell S$ is still a double of a ribbon surface, and $b_2(\pi_1(\Sigma_2(\#_\ell S))) = \ell$.  Again, the Theorem applies.
\end{proof}

Before proving the Theorem, we do a warm-up proposition that applies more broadly and could be relevant for the Kinoshita conjecture.  
\begin{proposition}
Let $S$ be a knotted surface in $S^4$ other than a sphere.  If  $b_2(\pi_1(\Sigma_2(S))) > 0$, then $S$ is not the connected sum of a knotted 2-sphere and an unknotted surface.  
\end{proposition}
\begin{proof}
Suppose for a contradiction that $S = N \# U$ where $N$ is a knotted 2-sphere and $U$ is an unknotted surface.  Then $\Sigma_2(S)$ is a connected sum of $\Sigma_2(N)$ and $\Sigma_2(U)$, where the latter is a connected sum of $CP^2$'s, $\overline{CP^2}$'s, and $S^2 \times S^2$'s.  This implies that $\pi_1(\Sigma_2(N)) = \pi_1(\Sigma_2(S))$.  Recall that the branched double cover of a knotted 2-sphere in $S^4$ is always a rational homology 4-sphere, in particular $0 = b_2(\Sigma_2(N))$.

Since an Eilenberg-MacLane space for $\pi_1(\Sigma_2(N))$ can be built by attaching cells to $\Sigma_2(N)$ of index~$\ge 3$ we have that $ b_2(\Sigma_2(N))  \geq b_2(\pi_1(\Sigma_2(N)))$. We now have the following contradiction.
\[
0 = b_2(\Sigma_2(N))  \geq b_2(\pi_1(\Sigma_2(N))) = b_2(\pi_1(\Sigma_2(S)) > 0. \qedhere
\]
\end{proof}
The proof of the Theorem is similar to the proof of the Proposition,  but we need some analysis of the image of $\pi_2(\Sigma_2(N))$ in $H_2(\Sigma_2(N))$.

\begin{proof}[Proof of the Theorem]
Suppose that $S$ is a knotted surface of Euler characteristic $2-2k$ obtained from doubling a ribbon surface in $B^4$.  Then $b_2(\Sigma_2(S)) = 2k$. 
Suppose for a contradiction that there exists a knotted surface $S'$ with $\chi(S') > \chi(S)$ and unknotted surfaces $U, U'$ so that $S \# U = S' \# U'$.  By further stabilizing both $U$ and $U'$ by an unknotted $RP^2$ if needed, we can assume that both $U$ and $U'$ are sums of unknotted $RP^2s$. Let $\ell = b_2(\Sigma_2(U))$, $\ell_{\pm}=b^{\pm}(\Sigma_2(U))$, so $\ell=\ell_++\ell_-$. Define $\ell', \ell_+'$, and $\ell_-'$ similarly for $U'$.  By definition, $\ell' > \ell$.  
We will study the intersection form on $H_2(\Sigma_2(S \# U))$ restricted to the image of $\pi_2(\Sigma_2(S \# U))$ (and similarly for $S'\#U'$), and derive a contradiction by getting different restricted forms for the two different decompositions.  
 
Note that $\Sigma_2(U)$ is a connected sum of $CP^2$'s and $\overline{CP^2}$'s, so $\pi_1(\Sigma_2(S \# U)) = \pi_1(\Sigma_2(S))$.  Hence, $b_2(\pi_1(\Sigma_2(S \# U))) = k$ by hypothesis and $b_2(\Sigma_2(S \# U)) = 2k + \ell$. 
The Hopf exact sequence for path-connected spaces \cite{Hopf} applied to $S \# U$ gives:
\[
\pi_2(\Sigma_2(S \# U)) \to H_2(\Sigma_2(S \# U)) \to H_2(\pi_1(\Sigma_2(S \# U))) \to 0.
\] The sequence implies the image of $\pi_2(\Sigma_2(S \# U))$ in $H_2(\Sigma_2(S \# U))$ has rank $k + \ell$.  We can geometrically realize a full rank subgroup of $im(\pi_2(\Sigma_2(S \# U)))$ as follows.  Since $S$ is  obtained from doubling a ribbon surface, there is a rank $k$ subgroup coming from  $H_2(\Sigma_2(S))$ represented by pairwise disjoint, square-zero embedded spheres.  Orthogonal to this rank $k$ subgroup, there is a rank $\ell$ subgroup in $im(\pi_2(\Sigma_2(S \# U)))$ coming from $H_2(\Sigma_2(U))$ which contributes an $\ell_+[+1]\oplus \ell_-[-1]$ summand to the restricted intersection form. Thus the nondegenerate part of the restricted intersection form has rank $\ell$.

However, from the perspective of $S' \# U'$, there is a rank $\ell'$ subgroup of $im(\pi_2(\Sigma_2(S' \# U')))$ coming from spherical classes in $\Sigma_2(U')$ which contributes an $\ell'_+[+1]\oplus \ell'_-[-1]$ summand to the intersection form.  This means that the nondegenerate part of the restricted intersection form has rank at least $\ell'>\ell$, which is a contradiction.      
\end{proof}

\begin{remark}\label{rmk:rp2-split}
We can also show that the Klein bottles built above in the proof of the Corollary do not split as a connected sum of two (possibly knotted) projective planes.  
If one had such a splitting, then $\Sigma_2(S) = P_1 \# P_2$, where $P_1$ and $P_2$ are each integer homology $CP^2$'s.  Since $\sigma(\Sigma_2(S)) = 0$, without loss of generality, $Q_{P_1} = (+ 1)$ and $Q_{P_2} = (-1)$.  However, $T(2,3,7)$ does not split as a free-product of groups\footnote{One way to see $T(2,3,7)$ does not split as a free-product of groups is as follows.  Since $T(2,3,7)$ is a hyperbolic reflection group, the complement of a vertex in its  Cayley graph is connected.  However, the complement of a vertex in the Cayley graph of a free product of non-trivial groups is disconnected.}
, and so either $\pi_1(P_1) = T(2,3,7)$ or $\pi_1(P_2) = 0$ or vice versa.  We consider the former; the argument for the latter is similar.  In this case, since $P_2$ is simply-connected, $H_2(P_2)$ is in the image of $\pi_2$; in particular, we can find a square $-1$ class in the image of $\pi_2$.  Since $\Sigma_2(S)$ is a double, it has an orientation-reversing homeomorphism.  That means that we can also find a $+1$ class in $H_2(\Sigma_2(S))$ in the image of $\pi_2$.  This implies all of $H_2(X)$ is in the image of $\pi_2$, which contradicts the Hopf exact sequence.  We note that many of Yoshikawa's irreducible Klein bottles in \cite{Yoshikawa} satisfy the same type of indecomposability.
\end{remark}


We conclude by giving an alternate construction of surfaces satisfying the hypotheses of the Theorem; in particular their branched double cover will have fundamental group $T(2,3,7)$. This construction comes from a double of a ribbon surface more directly.  

Let $D$ denote the usual ribbon disk for $P(-2,3,7) \# -P(-2,3,7)$.  Attach a band $B$ to turn the $P(-2,3,7)$ summand to any pretzel knot $P(-2,3,7,n)$ guided by the arc $\alpha$ in the top left frame of  Figure~\ref{fig:branched}. Double the surface $D \cup B$ to get a knotted torus or Klein bottle $S$. The orientability of $S$ is determined by the parity of $n$.   
Because $S$ is a double it necessarily has normal Euler number 0.  We now analyze the branched double cover and its fundamental group.  

Notice that $\Sigma_2(D)$ is $(Y - B^3) \times I$, where $Y$ is the Seifert fibered space $S^2(0;1/2,-1/3,-1/7)$. We sketch a proof of this in Figure \ref{fig:branched}. In the branched cover $B$ lifts to a 2-handle attachment. Let $C$ denote the lift of the arc $\alpha$.  Hence $\pi_1(\Sigma_2(S))$ can be described as the quotient of $\pi_1(Y)$ by the normal closure of $C$.  From Figure~\ref{fig:branched} one can see that $C$ is a regular fiber of $Y$. 
Killing a regular fiber in the fundamental group of a Seifert fibred space gives the orbifold fundamental group of the base orbifold (see e.g. \cite[Lemma 3.2]{Scott}),  which in this case is $T(2,3,7)$.  Consequently, the fundamental group of $\Sigma_2(S)$ is $T(2,3,7)$. Since $S$ is a double of a ribbon surface, the Theorem applies.  

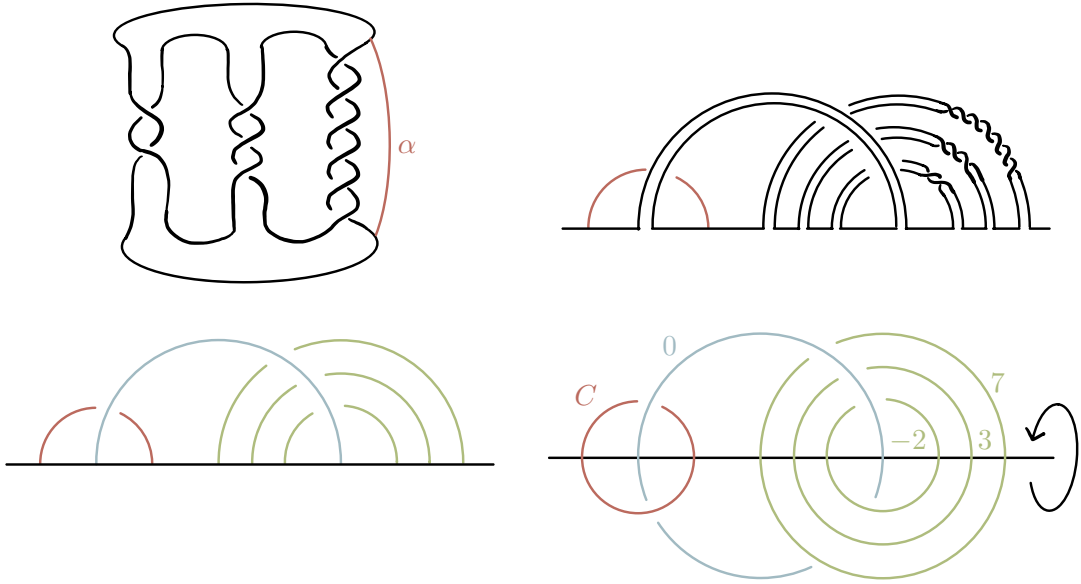
\begin{figure}
\begin{overpic}[scale=.6, trim = 0in 1.8in 0in 1.5in, clip]{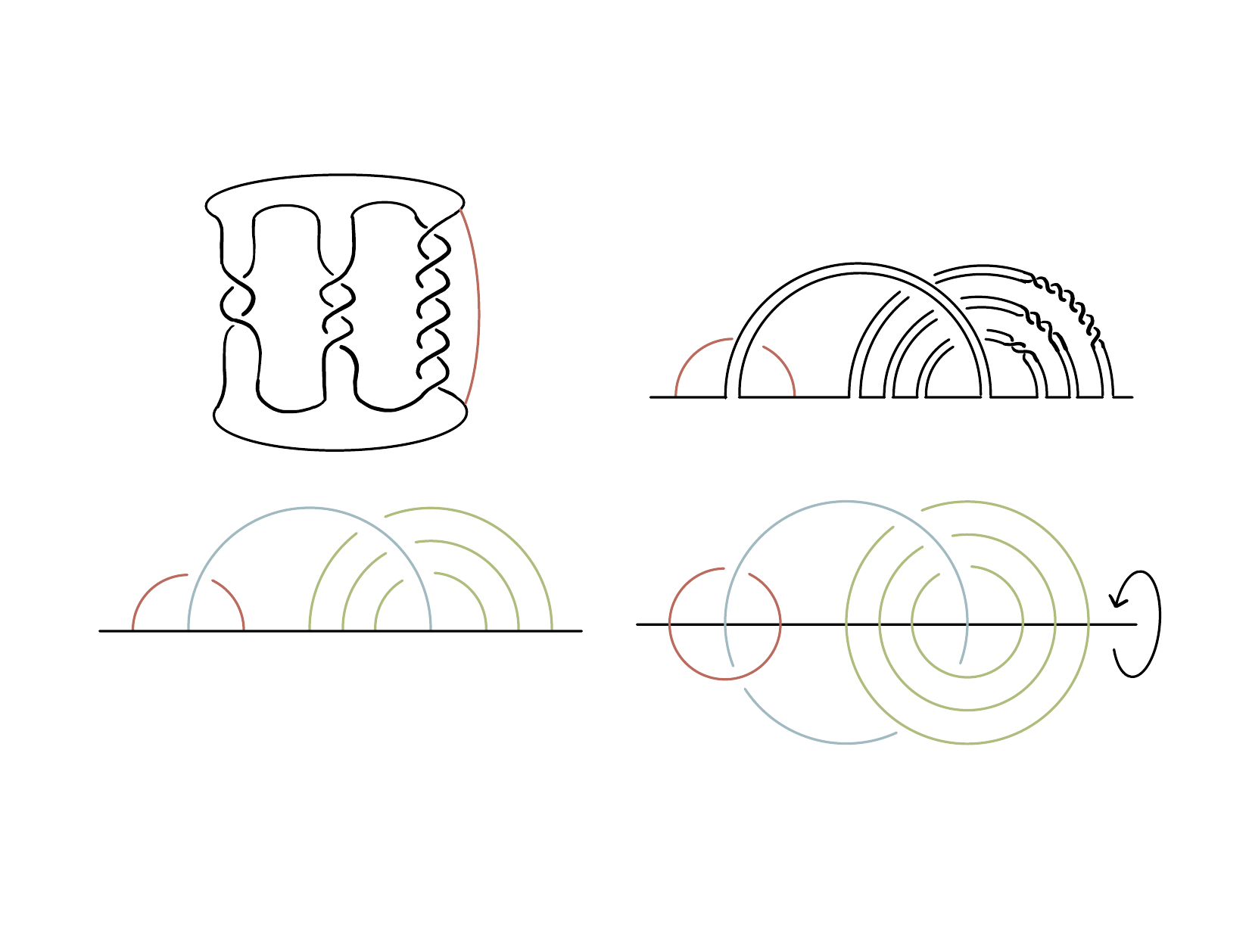}
 \put(39, 34){\color{lred}$\alpha$}
  \put(53, 14){\color{lred}$C$}
 \put(60, 18){\color{lblue}$0$}
  \put(86, 15){\color{lgreen}$7$}
    \put(85, 10.5){\color{lgreen}$3$}
       \put(78, 10.5){\color{lgreen}$-2$}
  \end{overpic}
\caption{Top left, we see the arc $\alpha$ for which band surgery goes from $P(-2,3,7)$ to $P(-2,3,7,n)$, where $n$ is the number of half-twists in the band.  Top right, we isotope $P(-2,3,7)\cup\alpha$ into a convenient position. Observe then that $P(-2,3,7)$ is obtained from the horizontal unknot by tangle replacements along the blue and green arcs marked in the bottom left. By the Montesinos trick, $\Sigma_2(P(-2,3,7))$ is obtained the double cover of the horizontal unknot by surgery along the blue and green link in the bottom right. Surgery on this link yields $S^2(0;1/2,-1/3,-1/7)$, and we see that $\alpha$ lifts to a regular fiber. For more details on passing between branch sets and surgery descriptions, we find \cite[Figure 19]{Bloom} helpful. 
}\label{fig:branched}
\end{figure}

\section*{Acknowledgements} TL was supported in part by NSF Grant DMS-2105469.  TL is grateful to SLMath where this project started on the side during the Summer Research in Mathematics program in Summer 2023, and to the UT Austin Department of Mathematics for their hospitality.  LP\ was supported in part by a Sloan Fellowship, a Clay Fellowship, and the Simons collaboration ``New structures in low-dimensional topology.'' We thank Anthony Conway, Jeff Danciger, Cameron Gordon, and Maggie Miller for helpful discussions.  

\bibliography{references}
\bibliographystyle{alpha}

\end{document}